\documentclass[11pt]{amsart}
\usepackage{color}

\usepackage{amsmath,amscd,amssymb,}
\newtheorem{theorem}{Theorem}[section]
\newtheorem{lemma}[theorem]{Lemma}

\newtheorem{remark}[theorem]{Remark}
\numberwithin{equation}{section}
\newtheorem{definition}[theorem]{Definition}

\newtheorem{corollary}[theorem]{Corollary}
\newtheorem{hypothesis}[theorem]{Hypothesis}

\input xy
\xyoption{all}
\begin{document}
 \begin{center}
 { \bf Some Estimates Regarding Integrated density of States for Random Schr\"{o}dinger Operator with decaying Random Potentials }
 \end{center}
\begin{center}
Dhriti Ranjan Dolai\\ The Institute of Mathematical Sciences\\ Taramani, Chennai - 600113, India\\
Email: dhriti@imsc.res.in
\end{center}
\clearpage

\begin{center}
{\bf Abstract} 
\end{center}
We investigate some bounds for the integrated density of states in the pure point regime for the random Schr\"{o}dinger 
operators with decaying random potentials, given by 
$H^{\omega}=-\Delta+\displaystyle\sum_{n\in\mathbb{Z}^d}a_nq_n(\omega)$, acting on $\ell^2(\mathbb{Z}^d)$, where $\{q_n\}_{n\in\mathbb{Z}^d}$ are i.i.d. random variables and $0<a_n\simeq|n|^{-\alpha},~~\alpha>0$.

\section{\bf Introduction}
\noindent The random Schr\"{o}dinger operator $H^\omega$ with decaying randomness on the Hilbert space $\ell^2(\mathbb{Z}^d)$ is  given by
\begin{equation}
\label{a}
H^{\omega}=-\Delta+V^{\omega},~~\omega \in \Omega.
\end{equation}
$\Delta$ is the adjacency operator defined by
$$(\Delta u)(n)=\displaystyle\sum_{|m-n|=1}u(m)~\forall~u\in\ell^2(\mathbb{Z}^d)$$
and 
\begin{equation}
\label{ran}
 V^{\omega}=\displaystyle\sum_{n\in\mathbb{Z}^d}a_nq_n(\omega)|\delta_n\rangle\langle\delta_n|,
\end{equation}
is the multiplication operator on $\ell^2(\mathbb{Z}^d)$ by the sequence $\{a_nq_n(\omega)\}_{n\in\mathbb{Z}^d}$. Here $\{\delta_n\}_{n\in\mathbb{Z}^d}$ is the standard basis for $\ell^2(\mathbb{Z}^d)$,
$\{a_n\}_{n\in\mathbb{Z}^d}$ is a sequence of positive real numbers such that $a_n\rightarrow 0$ as 
$|n|\rightarrow \infty$ and $\{q_n\}_{n\in\mathbb{Z}^d}$ are real valued iid random variables with an absolutely continuous 
probability distribution $\mu$ with bounded density. We realize $q_n$ as $\omega(n)$ on
$\big(\mathbb{R}^{\mathbb{Z}^d}, \mathcal{B}_{\mathbb{R}^{\mathbb{Z}^d}}, \mathbb{P})$, 
$\mathbb{P}=\bigotimes \mu$ constructed via Kolmogorov theorem. We refer to this probability space as 
$(\Omega,\mathcal{B},\mathbb{P})$ henceforth.\\
For any $B\subset\mathbb{Z}^d$ we consider the canonical orthogonal projection $\chi_B$ onto $\ell^2(B)$ and define the matrices
\begin{equation}
\label{sir1}
 H^{\omega}_B=\big(\langle\delta_n, H^{\omega}\delta_m\rangle\big)_{n,m\in B},~
G^B(z;n,m)=\langle\delta_{n},(H_B^{\omega}-z)^{-1}\delta_{m}\rangle,~G^B(z)=(H_B^{\omega}-z)^{-1}.
\end{equation}
$$
G(z)=(H^{\omega}-z)^{-1},~~G(z;n,m)=\langle\delta_{n},(H^{\omega}-z)^{-1}\delta_{m}\rangle, z\in\mathbb{C}^{+}.
$$
Note that $H^{\omega}_B$ is the matrix 
$$H^{\omega}_B=\chi_B H^{\omega}\chi_B~:~\ell^2(B)\longrightarrow\ell^2(B),~ a.e ~\omega.$$
One can note that the operators $\{H^{\omega}\}_{\omega\in\Omega}$ are self adjoint a.e $\omega$ and have a common core domain consisting of vectors with finite support.\\\\
Let $\Lambda_L$ denote the $d$-dimension box:
$$
\Lambda_L=\{(n_1,n_2,\cdots,n_d)\in\mathbb{Z}^d: |n_i|\leq L\}\subset\mathbb{Z}^d.
$$
We will work with the following hypothesis:
\begin{hypothesis} 
\label{hy}
(1) The measure $\mu$ is absolute continuous with density satisfies
\begin{equation}
\label{dc}
\rho(x) = \left\{
 \begin{array}{rl}
  0 & \text{if } |x|<1\\
    \frac{\delta-1}{2}~\frac{1}{|x|^{\delta}} & \text{if } |x|\ge1,~~forsome~\delta>1.
   \end{array} \right.
\end{equation}
(2) The sequence $a_n$ satisfy $a_n\simeq|n|^{-\alpha}$, $\alpha>0$.\\
(3) The pair $(\alpha,\delta)$ is chosen such that $d-\alpha(\delta-1)>0$ holds. This implies that $\beta_L\to\infty$ as $L\to\infty$, where $\beta_L$ is given by
\begin{align}
\label{od}
 \beta_L &=\sum_{n\in \Lambda_L} a_n^{\delta-1} \\
 &\simeq \displaystyle \sum_{n\in\Lambda_L}|n|^{-\alpha(\delta-1)} = O\bigg((2L+1)^{d-\alpha(\delta-1)}\bigg)\nonumber.
\end{align}
\end{hypothesis}
\begin{remark}
 \label{rem}
We have taken an explicit $\rho(x)$ in (\ref{dc}) in order to simplify the calculations in the proofs. Our results also hold for $\rho(x)=O\big(\frac{1}{|x|^\delta}\big),~~\delta>1
~~as~~|x|\to\infty$.
\end{remark}
\noindent In \cite{WMJ}, Kirsch-Krishna-Obermeit consider $H^{\omega}=-\Delta+V^{\omega}$ on $\ell^2(\mathbb{Z}^d)$ 
with the same $V^{\omega}$ as defined in (\ref{ran}). They showed that $\sigma(H^{\omega})=\mathbb{R}$ and 
$\sigma_c(H^{\omega})\subseteq[-2d,2d]$ a.e. $\omega$,
under some conditions on $\{a_n\}_{n \in \mathbb{Z}^d}$ and $\mu$ \big(The density of $\mu$ should not decay too fast at infinity and
$a_n$ should not decay too fast\big). For the precise condition on $a_n$'s and $\mu$ we recall
Definition 2.1 from \cite{WMJ}, which is given as follows.
\begin{definition}
 \label{a_nsuprt}
 Let $\{a_n\}$ be a bounded, positive sequence on $\mathbb{R}$. Then, $\big\{a_n\big\}-supp~\mu$ is defined by
 \begin{equation}
  \big\{a_n\big\}-supp~\mu:=\bigg\{x\in\mathbb{R}:\sum_n\mu\big(a_n^{-1}(x-\epsilon, x+\epsilon)\big)=\infty~\forall~\epsilon>0\bigg\}.
 \end{equation}
We call a probability measure $\mu$ asymptotically~large with respect to $a_n$ if\\ $\big\{a_{kn}\big\}-supp~\mu=\mathbb{R}$, for all $k\in\mathbb{Z}^+$.
\end{definition}
\noindent To show the existence of point spectrum outside $[-2d,2d]$ they verified Simon-Wolf
criterion \cite[Theorem 12.5]{BS} by showing exponential decay of the fractional moment of the Green function 
\cite[Lemma 3.2]{WMJ}. The decay is valid for $|n-m|>2R$ with energy $E\in\mathbb{R}\setminus[-2d,2d]$ and is given by
\begin{align}
\label{exdc}
\mathbb{E}^{\omega}(|G^{\Lambda_L}(E+i\epsilon:n,m)|^s)\leq D_{P(n.m)}e^{-c\big(\frac{|n-m|}{2}\big)},~~
E \in\mathbb{R}\setminus[-2d,2d],
\end{align}
where $\epsilon>0$, $0<s<1$, $c$ is a positive constant and $R\in \mathbb{Z}^+$. Here, $D_{P(n.m)}$ is a constant 
independent of $E$ and $\epsilon$, but polynomially bounded in $|n|$ and $|m|$.

\noindent Jak\v{s}i\'{c}-Last showed in \cite[Theorem 1.2]{JL} that for $d\ge3$, if $a_n\simeq |n|^{-\alpha}~~\alpha>1$ then
there is no singular spectrum inside $(-2d,2d)$ of $H^{\omega}$.

\noindent Here we take $(a_n,\mu)$ satisfying the condition given in \cite[Corollary 2.5]{WMJ} and Hypothesis \ref{hy}.
Then the spectrum of $H^{\omega}$ is $\mathbb{R}$ and $\sigma_c(H^{\omega})\subseteq [-2d,2d]$ a.e. $\omega$ (follows from \cite[Theorem 2.7]{WMJ}). 
We show that the average spacing of eigenvalues of $H^{\omega}_{\Lambda_L}$ near the energy $E\in \mathbb{R}\setminus[-2d,2d]$ are of order $\beta_L^{-1}$, whereas those close to $E\in[-2d,2d]$ have average spacing of the order $\frac{1}{(2L+1)^d}$.
This shows that the eigenvalues of $H^{\omega}_{\Lambda_L}$ are more densely distributed inside $[-2d,2d]$ where continuous part of spectrum of $H^{\omega}$ lies than the pure point regime which is outside $[-2d,2d]$.

\noindent We need following definitions before stating the results: 
\begin{equation}
\label{de1}
 N_L^{\omega}(E)=\#\big\{j:E_j\leq E,~E_j\in\sigma(H^{\omega}_{\Lambda_L})\big\},
\end{equation}
\begin{equation}
\label{de2}
 \tilde{N}_L^{\omega}(E)=\#\big\{j:E_j\ge E,~E_j\in\sigma(H^{\omega}_{\Lambda_L})\big\},
\end{equation}
\begin{equation}
 \label{de3}
\gamma_L(\cdot)=\frac{1}{\beta_L}\displaystyle
\sum_{n\in\Lambda_L}\mathbb{E}^{\omega}\big(\langle \delta_n, E_{H^{\omega}_{\Lambda_L}}(.)\delta_n\rangle\big).
\end{equation}
Our main results are as follows:
\begin{theorem}
 \label{th1}
If $E<-2d$ and $\epsilon=-2d-E>0$ then, we have 
$$\frac{1}{2}\frac{1}{(4d+\epsilon)^{(\delta-1)}}\leq
\varliminf_{L\to\infty}\frac{1}{\beta_L}\mathbb{E}^{\omega}(N^{\omega}_{L}(E))\leq
\varlimsup_{L\to\infty}\frac{1}{\beta_L}\mathbb{E}^{\omega}(N^{\omega}_{L}(E))
\leq \frac{1}{2}\frac{1}{\epsilon^{(\delta-1)}}.$$
For $E=2d+\epsilon>2d$ we have
$$\frac{1}{2}\frac{1}{(4d+\epsilon)^{(\delta-1)}}\leq
\varliminf_{L\to\infty}\frac{1}{\beta_L}\mathbb{E}^{\omega}(\tilde{N}^{\omega}_{L}(E))\leq
\varlimsup_{L\to\infty}\frac{1}{\beta_L}\mathbb{E}^{\omega}(\tilde{N}^{\omega}_{L}(E))
\leq \frac{1}{2}\frac{1}{\epsilon^{(\delta-1)}}.$$
\end{theorem}
\noindent Now we investigate the average number of eigenvalues of $H^{\omega}_{\Lambda_L}$ inside $[-2d,2d]$, which can be given as follows:
\begin{corollary}
 \label{cor3}
For any interval $(M_1,M_2)\varsupsetneq[-2d,2d]$ we have
\begin{equation}
 \label{l3}
\lim_{L\to\infty}\frac{1}{(2L+1)^d}\mathbb{E}^{\omega}\big(\# \big\{\sigma(H^{\omega}_{\Lambda_L})\cap (M_1,M_2)\big\}\big)=1.
\end{equation}
\end{corollary}
\begin{corollary}
\label{cor}
 If $M_1<-2d$ and $M_2>2d$ then, we have
\begin{equation}
 \label{vague}
\varlimsup_{L\to\infty}\gamma_L\big((-\infty,M_1]\cup[M_2,\infty)\big)\leq 
\frac{1}{2}\bigg[\frac{1}{(-2d-M_1)^{(\delta-1)}}+\frac{1}{(M_2-2d)^{(\delta-1)}}\bigg]
\end{equation}
For any interval $I\subseteq\mathbb{R}\setminus[-2d,2d]$ of length $|I|>4d$ there is a constant $C_I>0$ such that
\begin{equation}
 \label{nontri}
\varliminf_{L\to\infty}\gamma_L(I)\ge C_I>0.
\end{equation}
\end{corollary}
\begin{corollary}
\label{convague}
Let $M_1<-2d$ and $M_2>2d$ and $\gamma_L\upharpoonright_{(M_1,M_2)^c}$ denote the restriction of 
$\gamma_L$ to $\mathbb{R}\setminus(M_1,M_2)$.
The sequence of measure $\big\{\gamma_L\upharpoonright_{(M_1,M_2)^c}\big\}_L$ admits a subsequence which converges vaguely to a 
 non-trivial measure, say $\gamma$.
\end{corollary}
\noindent The above theorem give estimates for the average of $N_L^{\omega}(E)$ and $\tilde{N}_L^{\omega}(E)$,
but we can also get a point-wise estimate of the above quantities which is given by following theorem.
\begin{theorem}
 \label{th2}
For $d\ge2$, $0<\alpha<\frac{1}{2}$ and $1<\delta<\frac{1}{2\alpha}$ then for almost all $\omega$
$$\frac{1}{2}\frac{1}{(2d-E)^{(\delta-1)}}\leq \varliminf_{L\to\infty}\frac{1}{\beta_L}N^{\omega}_{L}(E)
\leq \varlimsup_{L\to\infty}\frac{1}{\beta_L}N^{\omega}_{L}(E)\leq \frac{1}{2}\frac{1}{(-2d-E)^{(\delta-1)}}~~for~E<-2d,$$
$$
\frac{1}{2}\frac{1}{(2d+E)^{(\delta-1)}}\leq\varliminf_{L\to\infty}\frac{1}{\beta_L}\tilde{N}^{\omega}_{L}(E)
\leq \varlimsup_{L\to\infty}\frac{1}{\beta_L}\tilde{N}^{\omega}_{L}(E)\leq \frac{1}{2}\frac{1}{(E-2d)^{(\delta-1)}}~~for~E>2d.
$$
\end{theorem}
~\\
\noindent In \cite{AAM}, Figotin-Germinet-Klein-M\"{u}ller studied the Anderson Model on $L^2(\mathbb{R}^d)$ with decaying random potentials given by
$$
H^{\omega}=-\Delta+\lambda\gamma_{\alpha}V^{\omega}~~on~~L^2(\mathbb{R}^d),
$$
where $\lambda>0$ is the disorder parameter and $\gamma_{\alpha}$ is the envelope function 
$$
\gamma_{\alpha}(x):=(1+|x|^2)^{-\frac{\alpha}{2}},~~ \alpha\ge0.
$$
They assumed that the density of the single site distribution is compact supported $L^{\infty}$ function. They showed that for $\alpha\in(0,2)$ the operator $H^{\omega}$ has infinitely many eigenvalues in $(-\infty,0)$ a.e. $\omega$. In \cite[Theorem 3]{AAM}, they
gave the bound for $N^{\omega}(E),~E<0$ (number of eigenvalues of $H^\omega$ below $E$) in terms of density of states for the stationary 
(i.i.d. case) Model.

\noindent In \cite{GJMS}, Gordon-Jaksi\'{c}-Molchanov-Simon studied the Model given by 
$$H^{\omega}=-\Delta+\displaystyle\sum_{n\in\mathbb{Z}^d}(1+|n|^{\alpha})q_n(\omega),~~\alpha>0~~on~~\ell^2(\mathbb{Z}^d),$$
where $\{q_n\}$ are i.i.d. random variables uniformly distributed on $[0,1]$.
They showed that if $\alpha>d$ then $H^{\omega}$ has discrete spectrum a.e. $\omega$. For the case when $\alpha\leq d$ they construct
a strictly decreasing sequence $\{a_k\}_{k\in\mathbb{N}}$ of positive numbers such that if $\frac{d}{k}\ge\alpha>\frac{d}{k+1}$
then for a.e. $\omega$ we have the following:\\
(i) $\sigma(H^{\omega})=\sigma_{pp}(H^{\omega})$ and eigenfunctions of $H^{\omega}$ decay exponentially,\\
(ii) $\sigma_{ess}(H^{\omega})=[a_k,\infty)$ and\\
(iii) $\#\sigma_{disc}(H^{\omega})<\infty$.\\
They also showed that \\
(a) If $\frac{d}{k}>\alpha>\frac{d}{k+1}$ and $E\in(a_j,a_{j-1})$, $1\leq j\leq k$, then
$$
\lim_{L\to\infty}\frac{N^{\omega}_L(E)}{L^{d-j\alpha}}=N_j(E)
$$
exists for a.e. $\omega$ and is a non random function.\\
(b) If $\alpha=\frac{d}{k}$ and $E\in(a_j,a_{j-1})$, $1\leq j<k$ the above is valid. If $E\in(a_k,a_{k-1})$ then
$$
\lim_{L\to\infty}\frac{N^{\omega}_L(E)}{lnL}=N_k(E)
$$
exists for a.e. $\omega$ and is a non random function.

\noindent In this work, we essentially show that for decaying potentials the confinement length is $(2L+1)^d$ inside $[-2d,2d]$ and $\beta_L$ 
outside $[-2d,2d]$. On the other hand, for the growing potentials (as in \cite{GJMS}), the confinement length is a function of energy.
\section{\bf On the pure point and continuous spectrum}
\noindent In this section, we work out the spectrum of $H^{\omega}$ under the Hypothesis \ref{hy}. Here we use \cite[Corollary 2.5]{WMJ} and \cite[Theorem 2.3]{WMJ}\\\\
Let $x<0$ and $\epsilon>0$ such that $x+\epsilon<0$ then, for large enough $|n|\ge M$ we have $a_n^{-1}(x+\epsilon)\leq-1$ since
$a_n^{-1}\to\infty$ as $|n|\to\infty$. For $|n|\ge M$ we have
\begin{align*}
\mu\bigg(\frac{1}{a_n}(x-\epsilon,x+\epsilon)\bigg) &= 
\int_{a_n^{-1}(x-\epsilon)}^{a_n^{-1}(x+\epsilon)}\rho(t)dt\\
 &=a_n^{(\delta-1)}\frac{\delta-1}{2}\int_{x-\epsilon}^{x+\epsilon}\frac{1}{|t|^{\delta}}dt~~~~~(using~\ref{dc}).
\end{align*}
Hence,
\begin{equation}
\label{an1}
 \sum_{n\in\mathbb{Z}^d}\mu\bigg(\frac{1}{a_n}(x-\epsilon,x+\epsilon)\bigg)\ge
\frac{\delta-1}{2}\int_{x-\epsilon}^{x+\epsilon}\frac{1}{|t|^{\delta}}dt\sum_{|n|\ge M}a_n^{(\delta-1)}=\infty,
\end{equation}
since $\beta_L=\displaystyle\sum_{n\in\Lambda_L}a_n^{(\delta-1)}\to\infty$ as $L\to\infty$ (using \ref{od}).\\
For $x>0$, a similar calculation will give 
\begin{equation}
\label{an2}
\sum_{n\in\mathbb{Z}^d}\mu\bigg(\frac{1}{a_n}(x-\epsilon,x+\epsilon)\bigg)=\infty,~~\epsilon>0.
\end{equation}
Now let $\epsilon>0$, there exist $M$ such that $a_n^{-1}\epsilon>1$ for $|n|\ge M$.
So, we have
\begin{align*}
\sum_{n\in\mathbb{Z}^d}\mu\bigg(\frac{1}{a_n}(-\epsilon,\epsilon)\bigg) &\ge \sum_{|n|\ge M}\mu(-a_n^{-1}\epsilon,a_n^{-1}\epsilon)\\
 &=2\sum_{|n|\ge M}\frac{\delta-1}{2} \int_1^{a_n^{-1}\epsilon}\frac{1}{t^\delta}dt\\
 &=\sum_{|n|\ge M}(1-\epsilon^{1-\delta}a_n^{\delta-1}).
\end{align*}
Since, $\displaystyle\sum_{n\in\Lambda_L}(1-\epsilon^{1-\delta}a_n^{\delta-1})\approx
\big[(2L+1)^d-(2L+1)^{d-\alpha(\delta-1)}\big]\to\infty$ as $L\to\infty$, it follows that
\begin{equation}
\label{an3}
\sum_{n\in\mathbb{Z}^d}\mu\bigg(\frac{1}{a_n}(-\epsilon,\epsilon)\bigg)=\infty.
\end{equation}
If $0<\epsilon_1<\epsilon_2$ then, we have
$$
\mu\bigg(a_n^{-1}(x-\epsilon_1,x+\epsilon_1)\bigg)\leq \mu\bigg(a_n^{-1}(x-\epsilon_2,x+\epsilon_2)\bigg)~~\forall~~x\in\mathbb{R}.
$$
Using the above inequality together with (\ref{an1}), (\ref{an2}) and (\ref{an3}) we have,
\begin{equation}
 \label{ans}
\sum_{n\in\mathbb{Z}^d}\mu\bigg(a_n^{-1}(x-\epsilon,x+\epsilon)\bigg)=\infty, \mbox{~for~all~} x\in\mathbb{R}\ \&\ \epsilon>0.
\end{equation}
Then, using (\ref{ans}) from \cite[Definition 2.1]{WMJ} we see that
\begin{equation*}
 M=\cap_{k\in\mathbb{Z}^{+}}(a_{kn}-supp~\mu)=\mathbb{R}.
\end{equation*}
Therefore, \cite[Corollary 2.5]{WMJ} and \cite[Theorem 2.3]{WMJ} will give the following description about the spectrum of $H^{\omega}$.
$$\sigma_{ess}(H^{\omega})=[-2d,2d]+\mathbb{R}=\mathbb{R}~and~\sigma_c(H^{\omega})\subseteq[-2d,2d]~a.e~\omega.$$

\section{ \bf Proof of main results}
\noindent {\bf Proof of Theorem \ref{th1}.}
~\\Define
$$
A^{\omega}_{L,\pm}=\pm 2d +\displaystyle\sum_{n\in\Lambda_L}a_nq_n(\omega)P_{\delta_n}.
$$
and
$$
N^{\omega}_{\pm, L}(E)=\#\{j;~E_j\leq E,~E_j\in \sigma(A^{\omega}_{L,\pm})\},~
N^{\omega}_{L}(E)=\#\{j:E_j\leq E,~E_j\in \sigma(H^{\omega}_{\Lambda_L})\}.
$$
Since $\sigma(\Delta)=[-2d,2d]$, following operator inequality
\begin{equation}
 \label{a1}
A^{\omega}_{L,-}\leq H^{\omega}_{\Lambda_L}\leq A^{\omega}_{L,+}.
\end{equation}
is there, with
$$
H^{\omega}_{\Lambda_L}=\chi_{\Lambda_L}\Delta\chi_{\Lambda_L}+\displaystyle\sum_{n\in\Lambda_L}a_nq_n(\omega)P_{\delta_n}.
$$
A simple application of the min-max principle \cite[Theorem 6.44]{KT} shows that
\begin{equation}
\label{a2}
 N^{\omega}_{+,L}(E)\leq N^{\omega}_{L}(E)\leq N^{\omega}_{-,L}(E).
\end{equation}
Now, the spectrum $\sigma(A^{\omega}_{L,\pm})$ of $A^{\omega}_{L,\pm}$ consists of only eigenvalues and is given by
$$
\sigma(A^{\omega}_{L,\pm})=\{n\in\Lambda_L: \pm 2d+a_nq_n(\omega)\}.
$$
Let $E<-2d$ with $E=-2d-\epsilon$, for some $\epsilon>0$. Then,
\begin{align}
\label{a4}
 N^{\omega}_{-,L}(E)&=\#\{n\in\Lambda_L: -2d+a_nq_n(\omega)\leq-2d-\epsilon\}\\
  &=\#\{n\in\Lambda_L: q_n(\omega)\in(-\infty,-a_n^{-1}\epsilon]\}\nonumber\\
 &=\sum_{n\in\Lambda_L}\chi_{_{\{\omega:q_n(\omega)\in(-\infty,-a_n^{-1}\epsilon]\}}}.\nonumber
\end{align}
Since $q_n$ are i.i.d, if we take expectation of both sides of (\ref{a4}) we get
\begin{align}
\mathbb{E}^{\omega}(N^{\omega}_{-,L}(E))&=\sum_{n\in\Lambda_L}\mu(-\infty, -a_n^{-1}\epsilon]\\
&=\sum_{n\in\Lambda_L}\int_{-\infty}^{-a_n^{-1}\epsilon}\rho(x)dx.\nonumber
\end{align}
Since $a_n^{-1}\to\infty$ as $|n|\to\infty$ and $\epsilon>0$, there exist an $M\in\mathbb{N}$ such that
$$
a_n^{-1}\epsilon>1,~-a_n^{-1}\epsilon<-1~~\forall~|n|>M.
$$
Therefore for large $L$, from (\ref{a4}) we get
\begin{align}
\label{ap0}
\mathbb{E}^{\omega}(N^{\omega}_{-,L}(E))&=\sum_{n\in\Lambda_L}\int_{-\infty}^{-a_n^{-1}\epsilon}\rho(x)dx\\
&=\sum_{n\in\Lambda_L,~|n|>M}\int_{-\infty}^{-a_n^{-1}\epsilon}\rho(x)dx+
\sum_{n\in\Lambda_L,~|n|\leq M}\int_{-\infty}^{-1}\rho(x)dx.
\end{align}
Since $\#\{n\in\mathbb{Z}^d:|n|\leq M\}\leq (2M+1)^d$, we have
\begin{align}
 \label{ap}
\sum_{n\in\Lambda_L,~|n|\leq M}\int_{-\infty}^{-1}\rho(x)dx &\leq(2M+1)^d\int_{-\infty}^{-1}\rho(x)dx\\
&=(2M+1)^d ~~\frac{\delta-1}{2}\int_{-\infty}^{-1}\frac{1}{|x|^{\delta}}dx\nonumber\\
&=\frac{(2M+1)^d}{2}\nonumber,~~\delta>1~is~given.
\end{align}
using \eqref{od} on (\ref{ap}) we have
\begin{equation}
 \label{apl1}
\lim_{L\to\infty}\frac{1}{\beta_L}
\sum_{n\in\Lambda_L,~|n|\leq M}\int_{-\infty}^{-1}\rho(x)dx=0.
\end{equation}
Now, 
\begin{align}
\label{ap1}
 \sum_{n\in\Lambda_L,~|n|>M}\int_{-\infty}^{-a_n^{-1}\epsilon}\rho(x)dx&=
\sum_{n\in\Lambda_L,~|n|>M}a_n^{-1}\int_{-\infty}^{-\epsilon}\rho(a_n^{-1}t)dt\\
&=\sum_{n\in\Lambda_L,~|n|>M}a_n^{(\delta-1)}\frac{\delta-1}{2}\int_{-\infty}^{-\epsilon}\frac{1}{|t|^{\delta}}dt\nonumber\\
&=\frac{\epsilon^{1-\delta}}{2}\sum_{n\in\Lambda_L,~|n|>M}a_n^{(\delta-1)},~\delta>1.\nonumber
\end{align}
This equality gives,
\begin{equation}
 \label{apl2}
\lim_{L\to\infty}\frac{1}{\beta_L}
\sum_{n\in\Lambda_L,~|n|>M}\int_{-\infty}^{-a_n^{-1}\epsilon}\rho(x)dx=\frac{\epsilon^{1-\delta}}{2}.
\end{equation}
Using (\ref{apl1}) and (\ref{apl2}) in (\ref{ap0}), we have
\begin{equation}
 \label{apl3}
\lim_{L\to\infty}\frac{1}{\beta_L}\mathbb{E}^{\omega}(N^{\omega}_{-,L}(E))=\frac{\epsilon^{1-\delta}}{2}=\frac{1}{2~\epsilon^{(\delta-1)}}>0.
\end{equation}
A similar calculation with $\mathbb{E}^{\omega}(N^{\omega}_{+,L}(E))$ gives,
\begin{equation}
\label{apl4}
\lim_{L\to\infty}\frac{1}{\beta_L}\mathbb{E}^{\omega}(N^{\omega}_{+,L}(E))=\frac{(4d+\epsilon)^{1-\delta}}{2}=
\frac{1}{2~(4d+\epsilon)^{(\delta-1)}}>0.
\end{equation}
Now, using (\ref{apl3}) and (\ref{apl4}) from (\ref{a2}), we conclude the inequality
\begin{equation}
\label{r1}
\frac{1}{2} \frac{1}{(4d+\epsilon)^{(\delta-1)}}\leq
\varliminf_{L\to\infty}\frac{1}{\beta_L}\mathbb{E}^{\omega}(N^{\omega}_{L}(E))\leq
\varlimsup_{L\to\infty}\frac{1}{\beta_L}\mathbb{E}^{\omega}(N^{\omega}_{L}(E))
\leq \frac{1}{2}\frac{1}{\epsilon^{(\delta-1)}}.
\end{equation}
If we define
\begin{equation}
\label{tilde}
\tilde{N}^{\omega}_{\pm,L}(E)=\#\{j:E_j\ge E,~E_j\in \sigma(A^{\omega}_{L\pm})\}, ~
\tilde{N}^{\omega}_{L}(E)=\#\{j:E_j\ge E,~E_j\in \sigma(H^{\omega}_{\Lambda_L})\}
\end{equation}
then the Min-max theorem and (\ref{a1}) together will give
\begin{equation}
\label{tilde1}
\tilde{N}^{\omega}_{-,L}(E)\leq \tilde{N}^{\omega}_{L}(E) \leq \tilde{N}^{\omega}_{+,L}(E).
\end{equation}
If $E=2d+\epsilon>2d$, for some $\epsilon>0$, a similar calculation results in
\begin{equation}
 \label{r2}
\frac{1}{2}\frac{1}{(4d+\epsilon)^{(\delta-1)}}\leq
\varliminf_{L\to\infty}\frac{1}{\beta_L}\mathbb{E}^{\omega}(\tilde{N}^{\omega}_{L}(E))\leq
\varlimsup_{L\to\infty}\frac{1}{\beta_L}\mathbb{E}^{\omega}(\tilde{N}^{\omega}_{L}(E))
\leq \frac{1}{2}\frac{1}{\epsilon^{(\delta-1)}}.
\end{equation}
The inequalities (\ref{r1}) and (\ref{r2}) together prove the Theorem \ref{th1}. \qed \\\\
{\bf Proof of Corollary \ref{cor3}:}\\
Since $H^{\omega}_{\Lambda_L}$ is a matrix of order $(2L+1)^d$, we have $\#\sigma(H^{\omega}_{\Lambda_L})=(2L+1)^d$. If
$M_1<-2d$ and $M_2>2d$ then,
\begin{equation}
 \label{count}
\#\bigg\{\sigma(H^{\omega}_{\Lambda_L})\cap(-\infty,M_1]\bigg\}~+
~\#\bigg\{\sigma(H^{\omega}_{\Lambda_L})\cap(M_1,M_2)\bigg\}~+~\#\bigg\{\sigma(H^{\omega}_{\Lambda_L})\cap[M_2,\infty)\bigg\}
=(2L+1)^d.
\end{equation}
Since
\begin{align}
\label{1}
 \frac{1}{(2L+1)^d}\mathbb{E}^{\omega}\big\{\sigma(H^{\omega}_{\Lambda_L})\cap(-\infty,M_1]\big\}&=
\frac{\beta_L}{(2L+1)^d}\frac{1}{\beta_L}\mathbb{E}^{\omega}(N^{\omega}_{L}(M_1)),
\end{align}
and from (\ref{r1}) and Hypothesis \ref{hy} we have 
$$\varlimsup_{L\to\infty}\frac{1}{\beta_L}\mathbb{E}^{\omega}(N^{\omega}_{L}(M_1))<\infty,~~and~~
~~\lim_{L\to\infty}\frac{\beta_L}{(2L+1)^d}=0,$$
the following limit holds
\begin{equation}
 \label{o1}
\lim_{L\to\infty}\frac{1}{(2L+1)^d}\mathbb{E}^{\omega}\big\{\sigma(H^{\omega}_{\Lambda_L})\cap(-\infty,M_1]\big\}=0.
\end{equation}
Similarly, using (\ref{r2}) we get
\begin{equation}
 \label{02}
\lim_{L\to\infty}\frac{1}{(2L+1)^d}\mathbb{E}^{\omega}\big\{\sigma(H^{\omega}_{\Lambda_L})\cap[M_2,\infty)\big\}=0.
\end{equation}
 Using the inequalities (\ref{count}), (\ref{o1}) and (\ref{02}), we see that for any interval $(M_1,M_2)$ containing $[-2d,2d]$ 
$$
\lim_{L\to\infty}\frac{1}{(2L+1)^d}\mathbb{E}^{\omega}\big(\# \big\{\sigma(H^{\omega}_{\Lambda_L})\cap (M_1,M_2)\big\}\big)=1.
$$
\qed\\
{\bf Corollary \ref{cor}:}\\
If $M_1<-2d$ then from $(\ref{de3})$ we have
\begin{align}
 \gamma_L(-\infty,M_1]&=\frac{1}{\beta_L}\mathbb{E}^{\omega}\bigg(Tr \big(E_{H^{\omega}_{\Lambda_L}}(-\infty,M_1]\big)\bigg)\\
&=\frac{1}{\beta_L}\mathbb{E}^{\omega}\big(N_L^{\omega}(M_1)\big)\qquad(using~~(\ref{de1})).\nonumber
\end{align}
This equality together with (\ref{r1}) gives
\begin{equation}
\label{u1}
 \varlimsup_{L\to\infty}\gamma_L(-\infty,M_1]\leq \frac{1}{2~(-2d-M_1)^{\delta-1}}\qquad(using~~\epsilon=-2d-M_1).
\end{equation}
Similarly, for $M_2>2d$, using (\ref{r2}), we get
\begin{equation}
 \label{u2}
 \varlimsup_{L\to\infty}\gamma_L[M_2,\infty)\leq \frac{1}{2~(M_2-2d)^{\delta-1}}~~~~~(using~~\epsilon=M_2-2d).
\end{equation}
(\ref{u1}) and (\ref{u2}) together proves (\ref{vague}).\\\\
Let $J=[E_1,E_2]\subset(-\infty,-2d)$ with $|J|>4d$, set
 $E_1=-2d-\epsilon_1$, $E_2=-2d-\epsilon_2$ such that $\epsilon_1-\epsilon_2>4d$. Then,
\begin{align}
 \gamma_L(J)&=\frac{1}{\beta_L}\mathbb{E}^{\omega}\big(N_L^{\omega}(E_2)\big)
-\frac{1}{\beta_L}\mathbb{E}^{\omega}\big(N_L^{\omega}(E_1)\big)\\
&\ge\frac{1}{\beta_L}\mathbb{E}^{\omega}\big(N_{+,L}^{\omega}(E_2)\big)-
\frac{1}{\beta_L}\mathbb{E}^{\omega}\big(N_{-,L}^{\omega}(E_1)\big)~~(using~~(\ref{a2})).\nonumber
\end{align}
Therefore, (\ref{apl4}) and (\ref{apl3}) give (\ref{nontri}), namely
$$
\varliminf_{L\to\infty}\gamma_L(J)\ge 
\frac{1}{2}\bigg[\frac{1}{(4d+\epsilon_2)^{(\delta-1)}}-\frac{1}{\epsilon_1^{(\delta-1)}}\bigg]>0.
$$
Similar result holds even when $J\subset(2d,\infty)$ with $|J|>4d$. \qed \\\\
{\bf Proof of Corollary \ref{convague}:}\\
From (\ref{vague}) we have
\begin{equation}
 \label{u3}
\sup_L\gamma_L\big((-\infty,M_1]\cup[M_2,\infty)\big)<\infty.
\end{equation}
We write $\mathbb{R}\setminus(M_1,M_2)=\bigcup_n A_n$, countable union of compact sets.
Now, $\gamma_L\upharpoonright_{A_n}$ (restriction of $\gamma_L$ to $A_n$) admits a weakly convergence subsequence by 
Banach-Alaoglu Theorem. Then, by a diagonal argument we select a subsequence of $\{\gamma_L\}$ which converges vaguely to a non-trivial
measure, say $\gamma$ on $\mathbb{R}\setminus(M_1,M_2)$.

\noindent The non-triviality of $\gamma$ is given by the fact that if $J\subset\mathbb{R}\setminus(M_1,M_2)$ is an interval such that
$4d<|J|<\infty$ then from (\ref{nontri}) we get
$$
\inf_L\gamma_L(J)>0.
$$
\qed \\
Before we proceed to the proof of Theorem \ref{th2}, we prove the following lemma.
\begin{lemma}
 \label{lem}
Let $\{X_n\}$ be sequence of random variables on a probability space $\big(\Omega, \mathcal{B}, \mathbb{P}\big)$ satisfying
$$
\sum_{n=1}^\infty\mathbb{P}\big(\omega: |X_n(\omega)-X(\omega)|>\epsilon\big)<\infty, ~~\epsilon>0.
$$
Then $X_n\xrightarrow{n\to\infty} X$ a.e. $\omega$.
\end{lemma}
\begin{proof} Define
$$
A_n(\epsilon)=\big\{\omega: |X_n(\omega)-X(\omega)|>\epsilon\big\}.
$$
If
$$
\sum_{n=1}^\infty\mathbb{P}\big(A_n(\epsilon)\big)=\sum_{n=1}^\infty\mathbb{P}\big(\omega: |X_n(\omega)-X(\omega)|>\epsilon\big)<\infty,
$$
then the Borel-Cantelli lemma gives
$$
\mathbb{P}\big(A(\epsilon)\big)=0, ~where~~A(\epsilon)=\bigcap_{n=1}^{\infty}\bigcup_{m=n}^\infty A_n(\epsilon).
$$
Now we have,
$$
\mathbb{P}\big(B(\epsilon)\big)=1~where~~B(\epsilon)=\bigcup_{n=1}^{\infty}\bigcap_{m=n}^\infty A_n(\epsilon)^c.
$$
For each $N\in\mathbb{N}$, we define
$$
B_N=B(1/N)~~and~~B=\bigcap_{N=1}^\infty B_N~ then ~\mathbb{P}(B)=1,~since~\mathbb{P}(B_N)=1.
$$
For any $\delta>0$, we can choose $M\in\mathbb{N}$ such that $\frac{1}{M}<\delta$. If $\omega\in B$ then, $\forall~N\in\mathbb{N}$ $\omega\in B_N$ From the construction of $B_M$, there exist a $K\in\mathbb{N}$ such that
$$
|X_m(\omega)-X(\omega)|\leq \frac{1}{M}<\delta~~\forall~m\ge K.
$$
So we have, 
$$
X_m\xrightarrow{m\to\infty} X~~on~~B~~with~~\mathbb{P}(B)=1.
$$
Hence the lemma. 
\end{proof}

\noindent {\bf Proof of Theorem \ref{th2}:}\\
Let $E=-2d-\epsilon$ for some $\epsilon>0$ and define
\begin{equation}
 \label{rand}
X_n(\omega):=\chi_{\{\omega: q_n(\omega)\leq -a_n^{-1}\epsilon\}}.
\end{equation}
Since $\{q_n\}_n$ are i.i.d., $\{X_n\}$ is a sequence of independent random variables. Now, from (\ref{a4}) we have
\begin{equation}
 \label{indran}
N^{\omega}_{-,L}(E)=\sum_{n\in\Lambda_L}X_n(\omega).
\end{equation}
We want to prove the following:
\begin{equation}
 \label{aecon}
\lim_{L\to\infty}\frac{N^{\omega}_{-,L}(E)-\mathbb{E}^{\omega}\big(N^{\omega}_{-,L}(E)\big)}{\beta_L}=0~~a.e~\omega.
\end{equation}
In view of Lemma \ref{lem}, in order to prove the above equation, it is enough to show 
\begin{equation}
 \label{usbo}
\sum_{L=1}^\infty\mathbb{P}\bigg(\omega:\frac{\big|N^{\omega}_{-,L}(E)-\mathbb{E}^{\omega}\big(N^{\omega}_{-,L}(E)\big)\big|}{\beta_L}
>\eta\bigg)<\infty~~\forall~~\eta>0.
\end{equation}
using Chebyshev's inequality we get
\begin{equation}
\label{dh}
\sum_{L=1}^\infty\mathbb{P}\bigg(\omega:\frac{\big|N^{\omega}_{-,L}(E)-\mathbb{E}^{\omega}\big(N^{\omega}_{-,L}(E)\big)\big|}{\beta_L}
>\eta\bigg)\leq\sum_{L=1}^\infty\frac{1}{\eta^2\beta_L^2} \mathbb{E}^{\omega}\bigg(N^{\omega}_{-,L}(E)-
\mathbb{E}^{\omega}\big(N^{\omega}_{-,L}(E)\big)\bigg)^2.
\end{equation}
We proceed to estimate the RHS of the above inequality.
\begin{align}
\mathbb{E}^{\omega}\bigg(N^{\omega}_{-,L}(E)-\mathbb{E}^{\omega}\big(N^{\omega}_{-,L}(E)\big)\bigg)^2 &=
\mathbb{E}^{\omega}\bigg(\sum_{n\in\Lambda_L}\big(X_n(\omega)-\mathbb{E}^{\omega}
\big(X_n(\omega)\big)\bigg)^2\nonumber\\
&=\sum_{n\in\Lambda_L}\mathbb{E}^{\omega}\bigg(X_n(\omega)-\mathbb{E}^{\omega}
\big(X_n(\omega)\big)\bigg)^2~~(X_n~~are~independent)\nonumber\\
&=\sum_{n\in\Lambda_L}
\bigg[\mathbb{E}^{\omega}(X_n^2)-\big(\mathbb{E}^{\omega}(X_n)\big)^2\bigg]\nonumber\\
&\leq\sum_{n\in\Lambda_L}\mathbb{E}^{\omega}(X_n^2)\nonumber\\
&=\sum_{n\in\Lambda_L}\mathbb{E}^{\omega}(X_n)\qquad\ (since\ X_n^2=X_n)\nonumber\\
&=\mathbb{E}^{\omega}\big(N^{\omega}_{-,L}(E)\big)\qquad(using~~(\ref{indran}))\nonumber
\end{align}
Using the above estimate in (\ref{dh}) we get,
\begin{align}
\label{fin}
& \sum_{L=1}^\infty\mathbb{P}\bigg(\omega:\frac{\big|N^{\omega}_{-,L}(E)-\mathbb{E}^{\omega}\big(N^{\omega}_{-,L}(E)\big)\big|}{\beta_L}
>\eta\bigg) \leq \frac{1}{\eta^2}\sum_{L=1}^\infty\frac{1}{\beta_L^2}\mathbb{E}^\omega\big(N^{\omega}_{-,L}(E)\big)\\
&\qquad\qquad=\frac{1}{\eta^2}\sum_{L=1}^\infty\frac{1}{\beta_L}\frac{1}{\beta_L}\mathbb{E}^{\omega}\big(N^{\omega}_{-,L}(E)\big)\nonumber\\
&\qquad\qquad\leq \frac{C}{\eta^2}\sum_{L=1}^\infty\frac{1}{\beta_L}\qquad\qquad(using~~(\ref{apl3}))\nonumber\\
&\qquad\qquad\backsimeq\sum_{L=1}^\infty\frac{1}{L^{d-\alpha(\delta-1)}}\qquad(using~(\ref{od})).\nonumber
\end{align}
As we have assumed in the theorem that $0<\alpha<\frac{1}{2}$, $1<\delta<\frac{1}{2\alpha}$ and $d\ge2$, we have $d-\alpha(\delta-1)>1$.
Thus, (\ref{usbo}) follows from (\ref{fin}).\\
Therefore, from (\ref{aecon}), for a.e. $\omega$, we have
\begin{align}
\label{l1}
 \lim_{L\to\infty}\frac{1}{\beta_L}N^{\omega}_{-,L}(E)&=
\lim_{L\to\infty}\frac{1}{\beta_L}\mathbb{E}^{\omega}\big(N^{\omega}_{-,L}(E)\big)\\
&=\frac{1}{2~\epsilon^{(\delta-1)}}~~(using~~(\ref{apl3}))\nonumber\\
&=\frac{1}{2~(-2d-E)^{(\delta-1)}}~~(E=-2d-\epsilon).\nonumber
\end{align}
A similar calculation gives, for a.e. $\omega$,
\begin{align}
\label{l2}
 \lim_{L\to\infty}\frac{1}{\beta_L}N^{\omega}_{+,L}(E)&=
\lim_{L\to\infty}\frac{1}{\beta_L}\mathbb{E}^{\omega}\big(N^{\omega}_{+,L}(E)\big)\\
&=\frac{1}{2~(4d+\epsilon)^{(\delta-1)}}~~(using~~(\ref{apl4}))\nonumber\\
&=\frac{1}{2~(2d-E)^{(\delta-1)}}~~(E=-2d-\epsilon).\nonumber
\end{align}
The inequalities (\ref{l1}), (\ref{l2}) together with (\ref{a2}) give,
for $E<-2d$ for a.e. $\omega$,
\begin{equation}
\label{2thr}
 \frac{1}{2}\frac{1}{(2d-E)^{(\delta-1)}}\leq \varliminf_{L\to\infty}\frac{1}{\beta_L}N^{\omega}_{L}(E)
\leq \varlimsup_{L\to\infty}\frac{1}{\beta_L}N^{\omega}_{L}(E)\leq \frac{1}{2}\frac{1}{(-2d-E)^{(\delta-1)}}.
\end{equation}
For $E>2d$ we compute $\tilde{N}^{\omega}_{\pm,L}(E)$ (as in (\ref{tilde})) exactly in the same way as give above. Thus, we can prove 
that, for a.e. $\omega$,
\begin{align*}
\lim_{L\to\infty}\frac{1}{\beta_L}\tilde{N}^{\omega}_{+,L}(E) &=
\lim_{L\to\infty}\frac{1}{\beta_L}\mathbb{E}^\omega\big(\tilde{N}^{\omega}_{+,L}(E)\big)\\
&=\frac{1}{2~(E-2d)^{(\delta-1)}}
\end{align*}
and
\begin{align*}
 \lim_{L\to\infty}\frac{1}{\beta_L}\tilde{N}^{\omega}_{-,L}(E) &=
\lim_{L\to\infty}\frac{1}{\beta_L}\mathbb{E}^\omega\big(\tilde{N}^{\omega}_{-,L}(E)\big)\\
&=\frac{1}{2~(2d+E)^{(\delta-1)}}.
\end{align*}
These equalities, together with (\ref{tilde1}) give the following.
For $E>2d$, a.e. $\omega$,
\begin{equation}
\label{lth}
\frac{1}{2} \frac{1}{(2d+E)^{(\delta-1)}}\leq\varliminf_{L\to\infty}\frac{1}{\beta_L}\tilde{N}^{\omega}_{L}(E)
\leq \varlimsup_{L\to\infty}\frac{1}{\beta_L}\tilde{N}^{\omega}_{L}(E)\leq \frac{1}{2}\frac{1}{(E-2d)^{(\delta-1)}}.
\end{equation}
\qed


\begin{thebibliography}{99}
\bibitem{PW} Anderson, P.W: \textsl{Absence of diffusion in certain random lattices}, Phys. Rev. {\bf 109}, 1492-1505, 1958.



\bibitem{AM} Aizenman, Michael; Molchanov, Stanislav: \textsl{Localization at large disorder and at extreme energies: 
an elementary derivation}, Commun. Math. Phys. {\bf157(2)}, 245-278, 1993.


\bibitem{AW} Aizenman, Michael; Warzel, Simone: \textsl{The Canopy Graph and Level Statistics for Random Operators on Trees},
Mathematical Physics, Analysis and Geometry, {\bf 9(4)}, 291-333, 2006.


\bibitem{CL} Carmona, Ren\'{e}; Lacroix, Jean: \textsl{Spectral theory of random Schrodinger operators}, Boston,
Birkhauser, 1990.

\bibitem{JFA} Combes, Jean-Michel; Germinet, Fran\c{c}ois; Klein, Abel: 
\textsl{Generalized Eigenvalue-Counting Estimates for the Anderson Model},
J Stat Physics {\bf 135(2)}, 201-216, 2009.




\bibitem{MM} Demuth, Michael; Krishna, M: \textsl{Determining Spectra in Quantum Theory},
 Progress in Mathematical Physics. {\bf 44}, Birkh\"{a}user, Boston, 2004.

\bibitem{DM} Dolai, Dhriti; Krishna, M:  \textsl{Level Repulsion for a class of decaying random potentials}, 
Markov Processes and Related Fields (to be appear), arXiv:1305.5619[math.SP].

\bibitem{DJ} Daley, D.J; Vere-Jones: \textsl{An Introduction to the Theory of Point Processes {\bf II},
General theory and structure}, Springer, New York, 2008.


\bibitem{AAM} Figotin, Alexander; Germinet, Fran\c{c}ois; Klein, Abel; M\"{u}ller, Peter: 
\textsl{Persistence of Anderson localization in Schr\"{o}dinger operators with decaying random potentials},
 Ark. Mat. {\bf 45(1)}, 15-30, 2007.

\bibitem{GK} Germinet, Fran\c{c}ois; Klopp, Fr\'{e}d\'{e}ric: \textsl{Spectral statistics for the discrete Anderson model in the 
localized regime}, 
Spectra of random operators and related topics, 11-24, RIMS Kôkyûroku Bessatsu, B27, Res. Inst. Math. Sci. (RIMS), Kyoto, 2011.

\bibitem{GMT} Gordon, A. Ya; Molchanov, S. A; Tsagani, B: 
\textsl{Spectral theory of one-dimensional Schrödinger operators with strongly fluctuating potentials}
 Funct. Anal. Appl. {\bf 25(3)}, 236-238, 1991.

\bibitem{GJMS}  Gordon, Y. A; Jak\v{s}i\'{c}, V; Mol\u{c}anov, S; Simon, B: 
\textsl{Spectral properties of random Schr\"{o}dinger operators with unbounded potentials},
Comm. Math. Phys. {\bf 157(1)}, 23-50, 1993.

\bibitem{JL} Jak\v{s}i\'{c}, Vojkan; Last, Yoram: \textsl{Spectral structure of Anderson type Hamiltonians}, Invent. Math,
 {\bf141(3)}, 561-577, 2000.

\bibitem{KT} Kato, Tosio: \textsl{Perturbation theory for Linear operators},
 Classics in Mathematics. Springer-Verlag, Berlin, 1995.



\bibitem{KR} Krishna M: \textsl{Anderson model with decaying randomness existence of extended
states}, Proc. Indian Acad. Sci. Math. Sci. {\bf 100}, 285-294 1990.

\bibitem{K} Krishna, M: \textsl{Continuity of intregrated density of states-independent randomness},
Proc. Ind. Acad. Sci. {\bf 117(3)}, 401-410, 2007.

\bibitem{KN} Killip, Rowan; Nakano, Fumihiko: \textsl{Eigenfunction Statistics in the Localized Anderson Model}, 
Ann. Henri Poincare {\bf 8(1)}, 27-36, 2007.

\bibitem{KotNak} Kotani, S; Nakano, Fumihiko:  \textsl{Level statistics of one-dimensional Schrödinger 
operators with random decaying potential}, Preprint, (2012).

\bibitem{WMJ} Kirsch, W; Krishna, M; Obermeit, J: \textsl{Anderson model with decaying randomness: mobility edge},
Math.Z. {\bf 235(3)}, 421-433, 2000.

\bibitem{KU} Kotani, S; Ushiroya, N:  \textsl{One-dimensional Schrodinger operators with random
decaying potentials}, Commun. Math. Phys. {\bf115(2)}, 247-266, 1988.

\bibitem{NM} Minami, Nariyuki: \textsl{Local Fluctuation of the Spectrum of a Multidimensional Anderson 
Tight Binding Model}, Commun. Math. Phys. {\bf 177(3)}, 709-725, 1996.




\bibitem{RS} Reed, Michael; Simon, Barry: \textsl{Method of modern mathematical physics {\bf I},
 Functional Analysis}, Academic Press, 1978.

\bibitem{BS} Simon, Barry: \textsl{Trace ideals and their applications},  Mathematical Surveys and Monographs, 120, 
American Mathematical Society, Providence, RI, 2005. viii+150 pp.


\bibitem{SW} Simon, Barry; Wolff, Tom: \textsl{Singular continuos spectrum under rank one perturbations and 
localization for random Hamiltonians}, Comm. Pure and Appl. Math.  {\bf 39(1)}, 75-90, 1986.
\end{thebibliography}
\end{document}